\documentclass[a4paper,reqno,12pt]{amsart}
\usepackage{t1enc}
\usepackage[margin=1.0in]{geometry}
\usepackage{ifthen}
\usepackage{graphicx}
\usepackage{xcolor}

\newcommand{\R}{\mathbf{R}}

\newcommand{\pr}{\mathbf{P}}
\newcommand{\ex}{\mathbf{E}}

\theoremstyle{plain}
\newtheorem{theorem}{Theorem}
\newtheorem{lemma}{Lemma}
\newtheorem{corollary}{Corollary}
\newtheorem{proposition}{Proposition}

\theoremstyle{definition}

\theoremstyle{remark}

\newcommand{\formula}[2][nolabel]
{\ifthenelse{\equal{#1}{nolabel}}
 {\begin{align*} #2 \end{align*}}
 {\ifthenelse{\equal{#1}{}}
  {\begin{align} #2 \end{align}}
  {\begin{align} \label{#1} #2 \end{align}}
 }
}

\numberwithin{equation}{section}

\begin{document}

%
%

\title[Fourier-Bessel heat kernel estimates]{Fourier-Bessel heat kernel estimates}
\author{Jacek Ma{\l}ecki, Grzegorz Serafin and Tomasz Zorawik}
\address{Jacek Ma{\l}ecki, Grzegorz Serafin \\ Faculty of Fundamental Problems of Technology \\ Department of Mathematics\\ Wroc{\l}aw University of Technology \\ ul.
Wybrze{\.z}e Wyspia{\'n}\-skiego 27 \\ 50-370 Wroc{\l}aw,
Poland
}
\email{jacek.malecki@pwr.edu.pl, grzegorz.serafin@pwr.edu.pl}
\address{Tomasz Zorawik \\ Faculty of Fundamental Problems of Technology \\ Department of Mathematics\\ Wroc{\l}aw University of Technology \\ ul.
Wybrze{\.z}e Wyspia{\'n}\-skiego 27 \\ 50-370 Wroc{\l}aw,
Poland
}
\email{tomasz.zorawik@pwr.edu.pl}


\thanks{The project was funded by the National Science Centre grant no. 2013/11/D/ST1/02622.}

\begin{abstract} 
We provide sharp two-sided estimates of the Fourier-Bessel heat kernel and we give sharp two-sided estimates of the transition probability density for the Bessel process in $(0,1)$ killed at $1$ and killed or reflected at $0$.
\end{abstract}

\maketitle
\section{Introduction}
\label{sec:introduction}
We consider the Fourier-Bessel heat kernel, which is represented in terms of the Bessel functions of the first kind $J_\nu(z)$ and its successive $n$-th positive zeros $\lambda_{n,\nu}$ in the following way
\formula[eq:FB:def]{
  G_t^\nu(x,y) = 2(xy)^{-\nu}\sum_{n=1}^\infty \exp\left(-\lambda_{n,\nu}^2 t\right)\frac{J_\nu(\lambda_{n,\nu}x)J_\nu(\lambda_{n,\nu}y)}{|J_{\nu+1}(\lambda_{n,\nu})|^2}\/,\quad x,y\in (0,1)\/,\quad t>0\/,
}
where $\nu>-1$. The main results of the paper are the following sharp two-sided estimates of $G_t^\nu(x,y)$ given in
\begin{theorem}
\label{thm:main}
For every $\nu>-1$ we have
\formula[eq:mainresult]{
   G_t^\nu(x,y) \stackrel{\nu}{\approx} \frac{(1+t)^{\nu+2}}{(t+xy)^{\nu+1/2}}\left(1\wedge\frac{(1-x)(1-y)}{t}\right)\,\frac{1}{\sqrt{t}}\exp\left(-\frac{|x-y|^2}{4t}-\lambda_{1,\nu}^2t\right)\/,
}
whenever $x,y\in(0,1)$ and $t>0$\/.
\end{theorem}
Here $\stackrel{\nu}{\approx}$ means that the ratio of the functions on the right and left-hand side is bounded from below and above by positive constants depending only on $\nu$. Since the sum in \eqref{eq:FB:def} is oscillating, this explicit representation can be only used to examine the behaviour of the kernel for large times. Indeed, it well-known that $G_t^\nu(x,y)$ behaves like the first term of the series, whenever $t\geq T_0>0$ (for every $T_0$, if we consider the upper-bounds and for some $T_0$, if we deal with the lower-bounds). However, the description of the behaviour of $G_t^\nu(x,y)$ for small times is very difficult to obtain from the above-given series representation, since the sum is highly oscillating and the cancellations between the terms matter in that case. This is a reason why we do not use \eqref{eq:FB:def} in examine the small-time behaviour, instead we explore the relation between the Fourier-Bessel heat kernel and the transition probability density  of the Bessel process with index $\nu$ reflected at $0$ and killed at $1$. This approach enables us to use probabilistic tools like, for example, the Hunt formula or the Strong Markov property, but still the purely analytic studies of the properties of the modified Bessel functions are crucial for the proofs. 

The Fourier-Bessel expansions naturally associated with the Fourier-Bessel heat kernel $G_t^\nu(x,y)$ has been studied for a long time in many different contexts, such as the study of the fundamental operators associated with the Fourier-Bessel expansions (see \cite{CiaurriRoncal:2005}, \cite{CiaurriRoncal:2010}, \cite{CiaurriStempak:2006b}, \cite{CiaurriStempak:2006a}, \cite{CiaurriStempak:2006c}) or the related Hardy spaces (\cite{DPRS:2015}) just to list a few from the latest works (see \cite{NowakRoncal:2014a} for more references). Moreover, the Fourier-Bessel expansions are successfully applicable in variety of areas outside Mathematics. The estimates of $G_t^\nu(x,y)$ has been recently studied in \cite{NowakRoncal:2014a} and \cite{NowakRoncal:2014b}, where the provided two-sided estimates of $G_t^\nu(x,y)$ were quantitatively sharp, i.e. the different constants appear in the exponential terms of the lower and upper bounds. It makes the estimates not sharp, whenever $|x-y|^2>>t$. In the estimates given in Theorem \ref{thm:main} the exponential behaviour of the kernel is described explicitly, i.e. the exponential terms in the lower and upper bounds are exactly the same. Such accurate results seem to be quite rare. Notice that even in the classical setting of Laplacian in $\R^n$, the known estimates of related Dirichlet heat kernel for smooth domains (see \cite{Zhang:2002}) are also only quantitatively sharp (see also \cite{SC:2010} and the references therein for corresponding results on manifolds). However, in the recent papers \cite{BogusMalecki:2015b} and \cite{BogusMalecki:2015a} the sharp two-sided estimates for the Dirichlet heat kernel of the half-line $(a,\infty)$ associated with the Bessel differential operator has been obtained. 

As we have previously mentioned, the result can be equivalently stated in the probabilistic context. More precisely, if we denote by $p_1^{(\nu)}(t,x,y)$ the transition probability density (with respect to the speed measure $m^{(\nu)}(dy)=y^{2\nu+1}dy$) of the Bessel process with index $\nu>-1$ killed at $1$ and reflected at $0$, then we have $p_1^{(\nu)}(2t,x,y) = G^{\nu}_{t}(x,y)$ and consequently
\begin{corollary}
For given $\nu>-1$ we have
\formula[eq:reflected:estimates]{
   p_{1}^{(\nu)}(t,x,y) \stackrel{\nu}{\approx} \frac{(1+t)^{\nu+2}}{(t+xy)^{\nu+1/2}}\left(1\wedge\frac{(1-x)(1-y)}{t}\right)\,\frac{1}{\sqrt{t}}\exp\left(-\frac{|x-y|^2}{2t}-\lambda_{1,\nu}^2t/2\right)\/,
}
whenever $x,y\in (0,1)$ and $t>0$.
\end{corollary}
Futhermore, instead of studying the Bessel process reflected at $0$, we can impose killing condition at both ends of the interval $(0,1)$. Then we can expend the range of the index of the process to the whole real line ($\mu\in \R$) and denote by $p_{(0,1)}^{(\mu)}(t,x,y)$ the transition probability density (with respect to the speed measure $m^{(\mu)}(dy)=y^{2\mu+1}dy$) of the corresponding process, i.e. the Bessel process killed when it leaves $(0,1)$. Note that for $\mu\geq 0$ the process does not hit $0$. Consequently the condition at zero (killing or reflecting) is relevant for the considered problem in that case, which means that $p_{(0,1)}^{(\mu)}(t,x,y)$ and $p_{1}^{(\nu)}(t,x,y)$ are identical for $\mu=\nu\geq 0$. Moreover, for $\mu<0$ we can use the absolute continuity property of the Bessel process with different indices to get that for every $\mu\geq 0$ we have
\formula{
   p^{(-\mu)}_{(0,1)}(t,x,y) = (xy)^{2\mu}p^{(\mu)}_{(0,1)}(t,x,y)\/,\quad x,y\in (0,1)\/,\quad t>0\/.
}
Collecting all together we obtain
\begin{corollary}
For given $\mu\in \R$ we have
\formula[eq:killed:estimates]{
   p_{(0,1)}^{(\mu)}(t,x,y) \stackrel{\mu}{\approx} \frac{(1+t)^{|\mu|+2}}{(t+xy)^{|\mu|+1/2}}\left[1\wedge\frac{(1-x)(1-y)}{t}\right]\,\frac{1\wedge (xy)^{-2\mu}}{\sqrt{t}}\exp\left(-\frac{|x-y|^2}{2t}-\frac{\lambda_{1,\nu}^2t}{2}\right)
}
whenever $x,y\in (0,1)$ and $t>0$.
\end{corollary}

The paper is organized as follows. In Preliminaries we collect some basic properties of the modified Bessel function of the first kind  $I_\nu(z)$ together with the estimates of the ratio of the form $I_\nu(y)/I_\nu(x)$. Then we introduce the basic notation together with some properties of the Bessel processes, which are used in the sequel. In Section \ref{sec:Proof} we provide the proof of Theorem \ref{thm:main}, which is divided into two parts. The first one relates to the lower bounds and the estimates in that case are given in Proposition \ref{prop:below}. The upper bounds are proved in Proposition \ref{prop:above}.

\section{Preliminaries}
\label{sec:preliminaries}
\subsection{Modified Bessel functions of the first kind} The modified Bessel function of the first kind is defined by 
\formula{
  I_\nu(z) = \sum_{k=0}^\infty \left(\frac{z}{2}\right)^{\nu+2k}\frac{1}{k!\Gamma(k+\nu+1)}\/,\quad z>0\/,\quad \nu>-1\/.
}
The above-given definition immediately implies that for every $\nu>-1$ we have
\formula[eq:I:asymp:zero]{
   I_\nu(z) \sim \left(\frac{z}{2}\right)^\nu \frac{1}{\Gamma(\nu+1)}\/,\quad z\to 0^+\/.
}
Moreover, the behaviour at infinity is described by
\formula[eq:I:asymp:infty]{
   I_\nu(z) \sim \frac{e^z}{\sqrt{2\pi z}}\/,\quad z\to \infty\/.
}
Here $\sim$ means that the corresponding limit of the ratio of the both functions is $1$. Finally, we will need the upper bounds of the ratio $I_\nu(y)/I_{\nu}(x)$, which can be found in \cite{Laforgia:1991}
\formula[eq:Laforgia]{
   \frac{I_\nu(y)}{I_\nu(x)}\leq e^{y-x}\left(\frac{y}{x}\right)^{\nu}\quad y\geq x>0\/,\quad \nu\geq-1/2\/.
}
However, the above-given result is not true for $\nu\in (-1,-1/2)$. Since we will need such kind of estimates only for large $x,y$, we introduce the following result valid for every $\nu>-1$ and large $x$ and $y$.  Note that the exponential term in the upper-bounds is the same as in (\ref{eq:Laforgia}), but the factor $(y/x)^{\nu}$ is here replaced by $(y/x)^{\nu+1}$.
\begin{lemma}
\label{lem:ratio:estimates} For every $y>x>1$ and $\nu>-1$ we have
\formula[eq:II:inequality]{
  \frac{I_\nu(y)}{I_\nu(x)}\leq e^{y-x}\left(\frac{y}{x}\right)^{\nu+1}\/.
}
\end{lemma}
\begin{proof}
  As we have mentioned, it is enough to consider $\nu\in (-1,-1/2)$ and the proof in this case will be a slight modification of that given in \cite{Laforgia:1991}. For every $z>0$ we have (see \cite{Nasell:1978})
	\formula{
	   \frac{I_{\nu+1}(z)}{I_\nu(z)}<\frac{z}{z+\nu+1/2}\/.
	}
	Thus, for every $z>1$, we can write
	\formula{
	  \frac{I_{\nu+1}(z)}{I_\nu(z)}<\frac{z}{z-1/2}\leq 1+\frac1z\/.
	}
	Consequently, the recurrent relation for the modified Bessel function implies
	\formula{
	   zI_\nu'(z) = \nu I_{\nu}(z)+zI_{\nu+1}(z)\leq (\nu+1)I_\nu(z)+zI_{\nu}(z)\/.
	}
	Dividing both sided of the above-given inequality by $zI_\nu(z)$ and integrating the obtained relation lead to 
	\formula{
	   \int_x^y \frac{I_\nu'(z)}{I_\nu(z)}\,dz\leq (\nu+1)\int_x^y \dfrac{dz}{z}+y-x\/,
	}
	whenever $y>x\geq 1$. This gives
	\formula{
	   \frac{I_\nu(y)}{I_\nu(x)}\leq \left(\frac{y}{x}\right)^{\nu+1}e^{y-x}\/,
	}
	which ends the proof.
	\end{proof}
\subsection{Bessel processes}
We write $\pr^{(\nu)}_x$ and $\ex^{(\nu)}_x$ for the probability law and the expected value of
a Bessel process with an index $\nu\in \R$ on the canonical path space with starting point $x\geq 0$. The filtration of the coordinate process $R(t)$ is denoted by $\mathcal{F}_t = \sigma\left\{R(s):s\leq t\right\}$. The transition density function (with respect to the speed measure $m^{(\nu)}(dy)=y^{2\nu+1}dy$) of the process (with reflecting condition impose on $0$, when $\nu\in(-1,0)$) is given in term of the modified Bessel function in the following way
\formula[eq:tdf:formula]{
   p^{(\nu)}(t,x,y) = \frac{(xy)^{-\nu}}{t}\exp\left(-\frac{x^2+y^2}{2t}\right)I_{\nu}\left(\frac{xy}{t}\right)\/,\quad x,y>0\/,t>0\/.
  }
	and
	\formula[eq:tdf:formula:zero]{
	  p^{(\nu)}(t,0,y) = \frac{1}{2^{\nu}t^{\nu+1}\Gamma(\nu+1)} \exp\left(-\frac{y^2}{2t}\right)\/,\quad y,t>0\/.
	}
	Taking into account the asymptotic behavior of $I_\nu(z)$ at zero \eqref{eq:I:asymp:zero} and at infinity \eqref{eq:I:asymp:zero} we obtain
	\formula[eq:tdf:estimate]{
	   p^{(\nu)}(t,x,y) \stackrel{\nu}{\approx} \frac{1}{(xy+t)^{\nu+1/2}}\,\frac{1}{\sqrt{t}}\exp\left(-\frac{(x-y)^2}{2t}\right)\/,
	}
	whenever $x,y>0$ and $t>0$.
	
Let us denote the first hitting time of a given point $a\geq 0$ by
	\formula{
	  T_a = \inf\{t>0: R(t)=a\}
	}
	and we write $q_{x,a}^{(\nu)}(ds)$ for the associated probability distribution with respect to $\pr_x^{(\nu)}$.
Moreover, we introduce the first exit time from the interval $(a,b)$ 
\formula[eq:Tab:dfn]{
   T_{(a,b)} = \inf\{t>0: R(t)\notin (a,b)\}\/, \quad 0\leq a<b \/.
}
The laws of Bessel processes with different indices are absolutely continuous and the corresponding Radon-Nikodym derivative is described by
\formula[ac:formula]{
\left.\frac{d\pr^{(\mu)}_x}{d\pr^{(\nu)}_x}\right|_{\mathcal{F}_t}=\left(\frac{R(t)}{x}\right)^{\mu-\nu}\exp\left(-\frac{\mu^{2}-\nu^2}{2}\int_{0}^{t}\frac{ds}{R^{2}(s)}\right)\/,
}
where $x>0$, $\mu,\nu\in\R$ and the above given formula holds $\pr^{(\nu)}_x$-a.s. on $\{T_0>t\}$.

The transition probability density function $p_1^{(\nu)}(t,x,y)$ of the process killed at $1$ can be expressed by the Hunt formula in the following way
\formula[eq:Hunt]{
   p_1^{(\nu)}(t,x,y) = p^{(\nu)}(t,x,y)-r_1^{(\nu)}(t,x,y)\/,
}
where 
\formula{
r^{(\nu)}_1(t,x,y) &= \ex^{(\nu)}_x[t<T_{1};p^{(\nu)}(t-T_1,R(T_{1}),y)]=\int_0^t p^{(\nu)}(t-s,1,y)\,q_{x,1}^{(\nu)}(ds)\/.
}
The last equality follows from the continuity of the path, i.e. the fact that $R(T_1)=1$ $P_x^{(\nu)}$-a.s..

Finally, we denote by $p_{(a,b)}^{(\mu)}(t,x,y)$ the transition probability density function of the process killed, when it leaves the interval $(a,b)$, for given $0\leq a<b$, i.e.
\formula{
   p_{(a,b)}^{(\mu)}(t,x,y) = \frac{\ex_x^{(\mu)}[t<T_{(a,b)};R(t)\in dy]}{m^{(\mu)}(dy)}\/,\quad x\in (a,b)\/,\quad t>0\/.
}
We will denote the index of the process by $\mu$, ($\mu\in\R$), when we deal with the process killed at $0$ (which is a case when $a=0$ above) to distinguish this situation from the case when $0$ is reflecting. Since the considered Bessel process with index $-1/2$ can be represented as the norm of one-dimensional Brownian motion, $p_{(0,1)}^{(-1/2)}(t,x,y)$ coincides with the corresponding object for Brownian motion on the real line, i.e. we have 
\formula[eq:-12:formula]{
  p_{(0,1)}^{(-1/2)}(t,x,y) = \frac{1}{\sqrt{2\pi t}}\sum_{n=-\infty}^\infty \left[\exp\left(-\frac{(x-y+2n)^2}{2t}\right)-\exp\left(-\frac{(x-y-2n)^2}{2t}\right)\right]\/.
 }

Moreover, the sharp two-sided estimates of $p_{(0,1)}^{({-1/2})}(t,x,y)$ are of the form (see Pyc, Serafin, Zak)
\formula[eq:-12:estimate]{
  p_{(0,1)}^{(-1/2)}(t,x,y) \approx \left(1\wedge\frac{xy}{t}\right)\left(1\wedge\frac{(1-x)(1-y)}{t}\right)\frac{1}{\sqrt{t}}\exp\left(-\frac{(x-y)^2}{2t}\right)\/,
}
whenever $x,y\in(0,1)$ and $t<1$. Using the scaling property and the shift-invariance of one-dimensional Brownian motion, we arrive at
\formula{
   p_{(a,b)}^{(-1/2)}(t,x,y) = \frac{1}{b-a}p_{(0,1)}^{(-1/2)}\left(\frac{t}{(b-a)^2},\frac{x-a}{b-a},\frac{y-a}{b-a}\right)\/,
}
for given $0\leq a<b$. It leads to
\formula[eq:-12:ab:estimate]{
p_{(a,b)}^{(-1/2)}(t,x,y) \approx\left(1\wedge\frac{(x-a)(y-a)}{t}\right)\left(1\wedge\frac{(b-x)(b-y)}{t}\right)\frac{1}{\sqrt{t}}\exp\left(-\frac{(x-y)^2}{2t}\right).
}
\section{Proof of Theorem \ref{thm:main}}
\label{sec:Proof}
The proof of Theorem \ref{thm:main} is divided into two parts, the first one relates to the lower bounds (Proposition \ref{prop:below}) and the other is devoted to show the upper bounds (Proposition \ref{prop:above}). Moreover, due to the symmetry, we will generally assume in the proofs in this section that $y>x$. 
\begin{proposition}
\label{prop:below}
For every $\nu>-1$ there exists $t_0=t_0(\nu)>0$ and constant $C_1^{(\nu)}>0$ such that
\formula[eq:estimate:below]{
   p_1^{(\nu)}(t,x,y)\geq C_1^{(\nu)} \left(1\wedge \frac{(1-x)(1-y)}{t}\right) \frac{1}{(xy+t)^{\nu+1/2}}\frac{1}{\sqrt{t}}\exp\left(-\frac{(x-y)^2}{2t}\right)\/,
}
whenever $x,y\in(0,1)$ and $t\leq t_0$.
\end{proposition}
\begin{proof}
   We generally assume that $t<1$ and we begin with the case when the space arguments are bounded away from $0$, i.e. $x,y\geq 1/32$. The constant $1/32$ is chosen for technical reasons. Since obviously $T_{(x/4,1)}\leq T_{1}$ we can write for any Borel set $A\subset (0,1)$ that
	\formula{
	  \int_A p_1^{(\nu)}(t,x,y)m^{(\nu)}(dy) = \ex^{(\nu)}_x [t<T_{1}; R(t)\in A]\geq \ex^{(\nu)}_x [t<T_{(x/4,1)}; R(t)\in A]\/.
	}
	Now, applying the absolute continuity property \eqref{ac:formula} we can see that the last expression is equal to 
	\formula[eq:acinproof]{
	  \ex^{(-1/2)}_x\left[t<T_{(x/4,1)}, R(t)\in A; \left(\frac{R(t)}{x}\right)^{\nu+1/2}\exp\left(-\frac{\nu^2-1/4}{t}\int_0^t \frac{ds}{R(s)^2}\right)\right]\/.
	}
	Notice that we can estimate the above-given integral functional on $\{t<T_{(x/4,1)}\}$ in the following way
	\formula{
	   \int_0^{t}\frac{ds}{R(s)^2} \leq \frac{16}{x^2} \int_0^t ds \leq 16\cdot 32^2 t \leq 16\cdot 32^2\/,
	}
	whenever $t<1$ and $x>1/32$ as we have assumed. Thus we have
	\formula{
	   \int_A p_1^{(\nu)}(t,x,y)m^{(\nu)}(dy)\geq c_0 \int_A \left(\frac{y}{x}\right)^{\nu+1/2}p_{(x/4,1)}^{(-1/2)}(t,x,y)\,dy
	}
	and consequently, by \eqref{eq:-12:ab:estimate}, we obtain
	\formula{
	   p_1^{(\nu)}(t,x,y)\geq \frac{c_1}{(xy)^{\nu+1/2}} \left(1\wedge\frac{3x(y-x/4)}{4t}\right)\left(1\wedge\frac{(1-x)(1-y)}{t}\right)\frac{1}{\sqrt{t}}\exp\left(-\frac{|x-y|^2}{2t}\right)\/,
	}
	for every $t<1$ and $x,y>1/32$. Observe also that since $y>x\geq 1/32$ and $t$ is bounded we have
	\formula{
	   1\wedge \frac{3x(y-x/4)}{4t}\approx 1 \approx \frac{1}{(t+xy)^{\nu+1/2}}\/,
	}
	which ends the proof in this case. 
	
	Now we assume that $x$ and $y$ are bounded away from $1$, i.e. $x,y\leq 1/4$. Recall that the subtrahend in the Hunt formula \eqref{eq:Hunt} is given by
	\formula{
	   r^{(\nu)}_1(t,x,y) &= \int_0^t p^{(\nu)}(t-s,1,y)\,q_{x,1}^{(\nu)}(ds)\/.
	}
	Using \eqref{eq:tdf:estimate} we obtain that for $\nu>-1/2$ we have
	\formula{
	  r^{(\nu)}_1(t,x,y) &\stackrel{\nu}{\approx} \int_0^t \frac{1}{(t-s+y)^{\nu+1/2}\sqrt{t-s}}\exp\left(-\frac{(1-y)^2}{2(t-s)}\right)\,q_{x,1}^{(\nu)}(ds)\\
		&\leq \int_0^t\frac{1}{(t-s)^{\nu+1}}\exp\left(-\frac{9}{32(t-s)}\right)\,q_{x,1}^{(\nu)}(ds)\/.
	}
	The function $t^{-\nu-1}\exp(-9/(32t))$ is increasing on $(0,t_1]$, where $t_1=9/(32(\nu+1))$. Thus, for $t<t_1$ we have
	\formula{
	   r_1^{(\nu)}(t,x,y)\leq \frac{1}{t^{\nu+1}}\exp\left(-\frac{9}{32t}\right)\int_0^t q_{x,1}^{(\nu)}(ds)\/.
		}
	Estimating the last integral simply by $1$ and using \eqref{eq:tdf:estimate} together with the fact that under our assumptions on $x$ and $y$ we have $|x-y|<1/4$, we can write
		\formula{
		   \frac{r^{(\nu)}(t,x,y)}{p^{(\nu)}(t,x,y)}\leq c_2(1+xy/t)^{\nu+1/2}\exp\left(-\frac{9}{32t}+\frac{|x-y|^2}{2t}\right)\leq c_3\frac{1}{t^{\nu+1/2}}\exp\left(-\frac{8}{32t}\right)\/.
		}
		Notice that for $\nu\in(-1,-1/2]$ and $t<9/16$ in a similar way we can arrive at		
		\formula{
		   r_1^{(\nu)}(t,x,y)&\stackrel{\nu}{\approx} \int_0^t \frac{(t-s+y)^{-\nu-1/2}}{\sqrt{t-s}}\exp\left(-\frac{(1-y)^2}{2(t-s)}\right)\,q_{x,1}^{(\nu)}(ds)\\
			&\leq (13/16)^{-\nu-1/2}\frac{1}{\sqrt{t}}\exp\left(-\frac{9}{32t}\right)
		}
		and consequently
		\formula{
		   \frac{r_1^{(\nu)}(t,x,y)}{p^{(\nu)}(t,x,y)}\leq c_4 (t+xy)^{\nu+1/2}\exp\left(-\frac{9}{32t}+\frac{|x-y|^2}{2t}\right)\leq c_4t^{\nu+1/2}\exp\left(-\frac{8}{32t}\right)\/.
		}
		Thus, in both cases ($\nu$ greater or smaller than $-1/2$), for $t$ sufficiently small we have
		\formula{
		   p_1^{\nu}(t,x,y) &= p^{(\nu)}(t,x,y)\left(1-\frac{r^{(\nu)}(t,x,y)}{p^{(\nu)}(t,x,y)}\right)\geq  \frac12 p^{(\nu)}(t,x,y)\/,\quad x,y<1/4\/.
		}
		Since $(1\wedge (1-x)(1-y)/t)\approx 1$ in this case, the usage of \eqref{eq:tdf:estimate} gives the result.
		
		Finally, we take $x\leq 1/8$ and $y\geq 1/4$. By the Chapmann-Kolmogorov equation we can write
		\formula{
		   p_1^{(\nu)}(t,x,y) &> \int_{1/32}^{1/4} p_1^{(\nu)}(t/8,x,z)p_1^{(\nu)}(7t/8,z,y)\,m^{(\nu)}(dz)\/.
		}
		Notice that we can use the previously obtained estimates since $x,z<1/4$ and $z,y>1/32$ in the integral above. Consequently, 
		\formula{
		  p_1^{(\nu)}(t,x,y) &> \frac{c_6}{t} \exp\left(-\frac{(x-y)^2}{2t}\right)K(t,x,y)\/,
		}
		where 
		\formula{
		   K(t,x,y) = \int_{1/32}^{1/4} \frac{1}{(t+xz)^{\nu+1/2}}\left(1\wedge\frac{(1-z)(1-y)}{t}\right)\exp\left(-\frac{(7x+y-8z)^2}{14t}\right)\,m^{(\nu)}(dz)\/.
		}
		Note that for $1/32\leq z\leq 1/4$, $x\leq 1/8$ and $y>1/4$ we have 
		\formula{
		   t+xz \approx t+xy\/,\quad \textrm{and} \quad (1-z)(1-y)/t\approx (1-x)(1-y)/t\/.
		}
		Moreover, since $z\geq 1/32$ we have
		\formula{
		   \int_{1/32}^{1/4}\exp\left(-\frac{(7x+y-8z)^2}{14t}\right)\,m^{(\nu)}(dz)&\stackrel{\nu}{\approx}  \int_{1/32}^{1/4}\exp\left(-\frac{(7x+y-8z)^2}{14t}\right)\,dz\\
			& = \sqrt{t}\int_{\frac{1-4(7x+y)}{32\sqrt{t}}}^{\frac{2-7x-y}{8\sqrt{t}}} \exp\left(-\frac{64w^2}{14}\right)\,dw\/,
		}
		where the last equality is obtained by substituting $8z-y-7x = 8w\sqrt{t}$. Finally note that $1/4\leq 7x+y\leq 15/8$ and consequently the lower bound of integration is non-positive and for $t\leq 1$ the upper bound is greater than $1/64$, which implies that the last integral given above can be estimated from below by a constant. Thus, combining all together we obtain that
		\formula{
		  K(t,x,y)\leq  \frac{c_7\,\sqrt{t}}{(t+xy)^{\nu+1/2}}\left(1\wedge\frac{(1-x)(1-y)}{t}\right)\/, 
		}
		which ends the proof.
	\end{proof}
\begin{proposition}
\label{prop:above}
For every $\nu>-1$ we can find $C_2^{(\nu)}>0$ and $t_0=t_0(\nu)>0$ such that
\formula{
   p_1^{(\nu)}(t,x,y)\leq C_2^{(\nu)} \left(1\wedge \frac{(1-x)(1-y)}{t}\right) \frac{1}{(xy+t)^{\nu+1/2}}\frac{1}{\sqrt{t}}\exp\left(-\frac{(x-y)^2}{2t}\right)\/,
}
for every $x,y\in(0,1)$ and $t\leq t_0$.
\end{proposition}
\begin{proof}
   The Hunt formula \eqref{eq:Hunt} together with \eqref{eq:tdf:estimate} immediately imply that 
	\formula{
	   p_1^{(\nu)}(t,x,y)\leq p^{(\nu)}(t,x,y)\stackrel{\nu}{\approx}  \frac{1}{(xy+t)^{\nu+1/2}}\frac{1}{\sqrt{t}}\exp\left(-\frac{(x-y)^2}{2t}\right)\/,
	}
	which gives the result in the case $(1-x)(1-y)\geq t$. Thus, from now on we will assume that $(1-x)(1-y)\leq t\leq 1/4$. Notice that it implies that $1-y\leq \sqrt{t}\leq 1/2$ and consequently $y\geq 1/2$. We split the proof into to parts. First, we assume that $x\leq 1/2$. The proof in this case is based on an idea to mimic the refection principle, which is true for Brownian motion and obviously is not for Bessel processes. However, it leads to correct estimates. More precisely, for every Borel set $A\subset (0,1)$ we write $\ex^{(\nu)}_x[t<T_{1};R(t)\in A]$ as
	 \formula{
	\left(\ex^{(\nu)}_x[R(t)\in A]-\ex^{(\nu)}_x[R(t)\in 2-A]\right)+ \left(\ex^{(\nu)}_x[R(t)\in 2-A]-\ex^{(\nu)}_x[t\geq T_{1};R(t)\in A]\right)
	}
	and denote the first part as $K_1(x,t,A)$ and the other as $K_2(x,t,A)$. Note that in the classical Brownian motion case, the second part vanishes and the reflection principle just gives the formula for the transition density function of the process killed at $1$. Here we obviously have to deal with both parts and consequently, we begin with the following estimates for $y\geq 1/2$ and $t<1/4$
	\formula[eq:proof:estimate1]{
	\nonumber
	   1-\frac{p^{(\nu)}(t,x,2-y)}{p^{(\nu)}(t,x,y)}\left(\frac{2-y}{y}\right)^{2\nu+1} &= 1-\left(\frac{2-y}{y}\right)^{\nu+1}\frac{I_\nu(x(2-y)/t)}{I_\nu(xy/t)}\exp\left(-\frac{2(1-y)}{t}\right)\\
		&\leq 1-\exp\left(-\frac{2(1-y)}{t}\right)\leq c_1\frac{1-y}{t}\/.
	}
	Here we have just simply used the monotonicity of $z^{\nu+1}$ and $I_\nu(z)$. Thus, we have 
	\formula{
	  K_1(x,t,A) = &\int_A\left(p^{(\nu)}(t,x,y)-p^{(\nu)}(t,x,2-y)\left(\frac{2-y}{y}\right)^{2\nu+1}\right)m^{(\nu)}(dy)\\
		&\leq c_1\frac{1-y}{t}\int_A p^{(\nu)}(t,x,y)m^{(\nu)}(dy)\/.
	}
	Moreover, for $y\geq 1/2$ and $t<1/4$ we can write
	\formula{
	   \left(\frac{2-y}{y}\right)^{2\nu+1}\frac{p^{(\nu)}(t,1,2-y)}{p^{(\nu)}(t,1,y)}-1 = \left(\frac{2-y}{y}\right)^{\nu+1}\frac{I_\nu((2-y)/t)}{I_\nu(y/t)}\exp\left(-\frac{2(1-y)}{t}\right)-1\/.
	}
	Now applying \eqref{eq:II:inequality} (note that $(2-y)/t>y/t>2$) we arrive at
	\formula[eq:proof:estimate2]{
	     \left(\frac{2-y}{y}\right)^{2\nu+1}\frac{p^{(\nu)}(t,1,2-y)}{p^{(\nu)}(t,1,y)}-1\leq \left(\frac{2-y}{y}\right)^{2\nu+1} -1\leq c_2(1-y)\/.
	}
   Since $A\subset(0,1)$, the condition $R(t)\in 2-A$ implies that $t\geq T_1$ and consequently
	\formula{
	  K_2(x,t,A) &=\ex_x^{(\nu)}[t\geq T_1;R(t)\in 2-A]-\ex_x^{(\nu)}[t\geq T_1;R(t)\in A]\/.
	}
	The Strong Markov property gives
	\formula{
	   K_2(x,t,A) &= \ex_x^{(\nu)}[t\geq T_1;H_1(t-T_1,A)] \/,
	}
	where
	\formula{
	 H_1(s,A) &= \ex_1^{(\nu)}\left[R(s)\in 2-A\right]-\ex_1^{(\nu)}\left[R(s)\in A\right]\\
	& = \int_{2-A}  p^{(\nu)}(s,1,y)m^{(\nu)}(dy)-\int_{A}  p^{(\nu)}(s,1,y)m^{(\nu)}(dy)\\
	&= \int_A\left(\left(\frac{2-y}{y}\right)^{2\nu+1}{p^{(\nu)}(t,1,2-y)}-{p^{(\nu)}(t,1,y)}\right)m^{(\nu)}(dy)\/.
	}
	The estimate provided in \eqref{eq:proof:estimate2} enable us to write
	\formula{
	   H_1(s,A)\leq c_2\int_A (1-y)p^{(\nu)}(s,1,y)m^{(\nu)}(dy)  = c_2\ex_1^{(\nu)}[R(s)\in A;1-R(s)]\/.
	}
	Now we can apply the Strong Markov property again to come back to
	\formula{
	   K_2(x,t,A)\leq c_2\int_A (1-y)p^{(\nu)}(t,x,y)m^{(\nu)}(dy)\leq c_2\int_A \frac{1-y}{t} p^{(\nu)}(t,x,y)m^{(\nu)}(dy)\/.
	}
	Combining all together leads to
	\formula[eq:upper:xsmall]{
	   p_1^{(\nu)}(t,x,y)\leq c_3\frac{1-y}{t}p^{(\nu)}(t,x,y)\/,\quad x\leq 1/2\/,\quad y\geq 1/2\/,\quad t\leq 1/4\/,
	}
	which ends the proof in this case.
	
	Now we assume that $x\in(1/2,1)$ and write for every Borel set $A\subset(0,1)$
	\formula[eq:midstep3]{
	   \ex_x^{(\nu)}[t<T_{1};R(t)\in A] = \ex_x^{(\nu)}[t<T_{(x/4,1)};R(t)\in A]+\ex_x^{(\nu)}[T_{(x/4,1)}\leq t<T_{1};R(t)\in A]\/.
	}
	Intuitively, the first part should be larger then the other one, since for $x,y\geq 1/2$ we should expect that there are more trajectories going from $x$ to $y$ which do not go below $x/4$, then those visiting level $x/4$ before reaching $y$. Indeed, note that the first term can be estimated in the same way as in the proof of Proposition \ref{prop:below}, i.e. we can apply the absolute continuity property, which together with the boundedness of the integral functional (for $x>1/2$ and $t\leq 1$) appearing in \eqref{eq:acinproof} imply
		\formula[eq:midstep2]{
		   \ex_x^{(\nu)}[t<T_{(x/4,1)};R(t)\in A] \leq c_4 \int_A \left(\frac{y}{x}\right)^{\nu+1/2}p_1^{(-1/2)}(t,x,y)\,dy\/.
		}
		Now it is enough to show that the second term in the right-hand side of \eqref{eq:midstep3} is significantly smaller then the expression on the left-hand side. To see this, we apply the Strong Markov property to write it as
		\formula{
		  \ex_x^{(\nu)}\left[T_{(x/4,1)}\leq t<T_{1};H_2(t-T_{(x/4,1)},x/4,A)\right]\/,
		}
		where 
		\formula{
		   H_2(u,x/4,A) = \ex_{R(T_{(x/4,1)})}^{(\nu)}[u<T_1;R(u)\in A] =  \ex_{x/4}^{(\nu)}[u<T_1;R(u)\in A]\/.
		  }
		The last equality follows from the fact that $R(T_{(x/4,1)})=x/4$ on $\{T_{(x/4,1)}\leq t<T_{1}\}$. Since $x/4<1/2$ we can apply the result given in \eqref{eq:upper:xsmall} to estimate $H_2(u,x/4,A)$ in the following way
		\formula{
		  \ex_{x/4}^{(\nu)}[u<T_1;R(u)\in A] \leq c_3 \int_A \left(1\wedge \frac{1-y}{u}\right)\frac{1}{\sqrt{u}}\exp\left(-(x/4-y)^2/(2u)\right) m^{(\nu)}(dy)
		}
		and let us denote the right-hand side above as $H_3(u,x/4,A)$. Notice that the function $u \rightarrow u^{-\alpha}\exp(-b^2/(2u))$ is increasing for $u<b^2/(2\alpha)$ and we have $|x/4-y|>1/4$, thus $u\to H_3(u,x/4,A)$ is increasing for $u\leq 1/48$. Moreover, we have
		\formula{
		   (y-x/4)^2-(y-x/2)^2\geq \frac{x}{4}(2y-3x/4)\geq \frac{x}{16}\geq \frac{1}{32}\/,\quad x,y\geq 1/2\/,
		}
		and consequently, for $u\leq t$ we obtain
		\formula{
		   H_3(u,x/4,A) \leq \exp\left(-\frac{1}{64 t}\right)H_3(u,x/2,A)\/.
		}
		Collecting all together with the fact that $T_{(x/2,1)}\leq T_{(x/4,1)}$ we can write for $t\leq 1/48$ that
		\formula{
		  \ex_x^{(\nu)}\left[T_{(x/4,1)}\right.\lefteqn{\left.\leq t<T_{1};H_2(t-T_{(x/4,1)},x/4,A)\right] }\\
			&\leq c_3 \ex_x^{(\nu)}\left[T_{(x/2,1)}\leq t<T_{1};H_3(t-T_{(x/2,1)},x/4,A)\right]\\
			&\leq c_3 \exp\left(-1/(64t)\right)\ex_x^{(\nu)}\left[T_{(x/2,1)}\leq t<T_{1};H_3(t-T_{(x/2,1)},x/2,A)\right]\\
			&\leq c_5 \exp\left(-1/(64t)\right)\ex_x^{(\nu)}\left[T_{(x/2,1)}\leq t<T_{1};H_2(t-T_{(x/2,1)},x/2,A)\right]\\
			& = c_5 \exp\left(-1/(64t)\right)\ex_x^{(\nu)}\left[T_{(x/2,1)}\leq t<T_{1};R(t)\in A\right]\/,
		}
		the last two lines follows from Proposition \ref{prop:below} and the Strong Markov property respectively. Since obviously
		\formula{
		   \ex_x^{(\nu)}\left[T_{(x/2,1)}\leq t<T_{1};R(t)\in A\right]\leq \ex_x^{(\nu)}\left[t<T_{1};R(t)\in A\right]
		}
		we can choose $t_0<1/48$ in such a way that for $t\leq t_0$ we have $c_5 \exp\left(-1/(64t)\right)\leq 1/2$ and consequently
		\formula{
		  \ex_x^{(\nu)}[t<T_{1};R(t)\in A] &= \ex_x^{(\nu)}[t<T_{(x/4,1)};R(t)\in A]+ \ex_x^{(\nu)}[T_{(x/4,1)}\leq t<T_{1};R(t)\in A] \\
			&\leq \ex_x^{(\nu)}[t<T_{(x/4,1)};R(t)\in A] +\frac{1}{2}\ex_x^{(\nu)}[t<T_{1};R(t)\in A]
		}
		and the estimates given in \eqref{eq:midstep2} ends the proof.
\end{proof}
Theorem \ref{thm:main} is now the consequence of the following. It is known that there exists $C_3^{(\nu)}>1$ such that
\formula[eq:largetime]{
   \frac{1}{C_3^{(\nu)}}(1-x)(1-y)\exp(-\lambda_{1,\nu}^2 t)\leq G_t^\nu(x,y)\leq C_3^{(\nu)}(1-x)(1-y)\exp(-\lambda_{1,\nu}^2 t)\/,
} 
for every $x,y\in (0,1)$, where the upper bounds holds for $t\geq T_0$ for arbitrary $T_0>0$, but in the lower bounds we have only the existence of such $T_0$ (see for example Theorem $3.7$ in \cite{NowakRoncal:2014a}). Thus, the upper bounds in \eqref{eq:mainresult} are just the consequence of \eqref{eq:largetime} and Proposition \ref{prop:above}. To finish the proof it is enough to show that the result of Proposition \ref{prop:below} is true for $t<T_0$ for arbitrary $T_0>0$. Let $t_0>0$ be as in the theses of Proposition \ref{prop:below} and note that \eqref{eq:estimate:below} reads as
\formula{
   p_1^{(\nu)}(t,x,y)\geq c_1(t_0,\nu) (1-x)(1-y)\/,\quad x,y\in (0,1)\/,
}
whenever $t\in [t_0/2,t_0]$. Thus, applying the Chapmann-Kolmogorov equation we have 
\formula{
  p_1^{(\nu)}(2t,x,y) &= \int_0^1 p_1^{(\nu)}(t,x,z)p_1^{(\nu)}(t,z,y)m^{(\nu)}(dz)\\
	& \geq (c_1(t_0,\nu))^2 (1-x)(1-y)\int_0^1 (1-z)^2z^{2\nu+1}\,dz\\
	&\geq c_3(t_0,\nu) \left(1\wedge \frac{(1-x)(1-y)}{t}\right) \frac{1}{(xy+t)^{\nu+1/2}}\frac{1}{\sqrt{t}}\exp\left(-\frac{(x-y)^2}{2t}\right)\/, 
}
whenever $t\in[t_0/2,t_0]$, which enable us to replace the condition $t<t_0$ in \eqref{eq:estimate:below} by $t<2t_0$ and in consequence by $t<T_0$ for arbitrary $T_0>0$. This ends the proof of Theorem \ref{thm:main}.

 \bibliography{bibliography}
\bibliographystyle{plain}
\end{document}